\documentclass{aims}
\usepackage{amsmath}
  \usepackage{paralist}
  \usepackage{graphics} 
  \usepackage{epsfig} 
\usepackage{graphicx}  \usepackage{epstopdf}
 \usepackage[colorlinks=true]{hyperref}
\usepackage[margin=2cm]{geometry}
\usepackage{latexsym}		
\usepackage{graphicx}
\usepackage{color}
\usepackage{subfigure}
\usepackage{amsmath,amsfonts,wasysym, amssymb}
\usepackage{bm}
\usepackage{longtable}
\usepackage{tabularx}
\usepackage{comment}
\usepackage{float}
\usepackage{tikz}
\usepackage{tikz-3dplot}
\usetikzlibrary{positioning}
\usetikzlibrary{shapes}
\usetikzlibrary{plotmarks}
\usetikzlibrary{shapes.geometric, calc}

\hypersetup{urlcolor=blue, citecolor=red}

\def\currentvolume{X}
 \def\currentissue{X}
  \def\currentyear{20XX}
   \def\currentmonth{XX}
    \def\ppages{X--XX}

\newtheorem{theorem}{Theorem}[section]
\newtheorem{corollary}{Corollary}

\newtheorem{lemma}[theorem]{Lemma}

\theoremstyle{definition}

\newtheorem{remark}{Remark}

\title[Some Estimates of Virtual Element Methods] 
      {Some Estimates of Virtual Element Methods for Fourth Order Problems}

\author[Qingguang Guan]{}


 \email{qg11b@my.fsu.edu}

\thanks{$^*$ Corresponding author: Qingguang Guan}

\begin{document}
	
\makeatletter
\newenvironment{equationate}{%
	\itemize
	\let\orig@item\item
	\def\item{\orig@item[]\refstepcounter{equation}\def\item{\hfill(\theequation)\orig@item[]\refstepcounter{equation}}}
}{%
	\hfill(\theequation)%
	\enditemize
}
\makeatother

\maketitle

\centerline{\scshape Qingguang Guan$^*$}
\medskip
{\footnotesize
\centerline{Department of Mathematics}
\centerline{Temple University}
\centerline{Philadelphia, PA 19122, USA}
} 

\medskip

\begin{abstract}
In this paper, we employ the techniques developed for second order operators to obtain the new estimates of Virtual Element Method for fourth order operators. The analysis is based on elements with proper shape regularity. Estimates for projection and interpolation operators are derived. Also, the biharmonic problem is solved by Virtual Element Method, optimal error estimates are obtained. Our choice of the discrete form for the right hand side function relaxes the regularity requirement in previous work and the error estimates between exact solutions and the computable numerical solutions are provided. 	
\end{abstract}

\section{Introduction}
Virtual Element methods are designed on polygonal/polyhedral meshes, see \cite{Beirao13}--\cite{Beir16-3}. It gives us the flexibility to use a wide range of meshes which is a great advantage especially in computational mechanics. And the computation cost is less than weak Galerkin finite element method, which can also employ general shape elements, see \cite{Wang14}--\cite{Guan18}.  Virtual Element Methods for second order problems are well studied in \cite{Beirao13}--\cite{Beir16-3}, also the stability and error analysis for these methods are obtained. New techniques based on the shape regularity and discrete norm for virtual element functions are developed in \cite{Brenner17}.  Virtual Element Methods for fourth order problems are analyzed in \cite{Brezzi13,Chinosi16}, however, the stability and error analysis are not completed. 

The motivations for this paper are: firstly, apply the techniques in \cite{Brenner17} to higher order problems, getting the basic estimates; secondly, to improve the error analysis for biharmonic equation. If we modify the virtual element method slightly then the regularity requirement for right hand side function $f$ can be relaxed;  thirdly, since the numerical solutions $u_h$ of virtual element methods can't be computed directly, to overcome this drawback, we present two ways to get the approximations of $u_h$ preserving the same convergence rate, and the computation of the approximations is much more efficient.

The paper is organized as follows: In Section 2, 2.1--2.3, the definition of two dimensional virtual element with the shape regularity is given. The projection in 2.2 is the same as in \cite{Chinosi16}. Compared with the definition of local virtual element space in \cite{Brezzi13,Chinosi16}, our method has the right hand side polynomial $q_v(= \Delta^2v) \in \mathbb{P}_k(D)$ not $\mathbb{P}_{k-4}(D)$ or $\mathbb{P}_{k-2}(D)$. This new definition provides more degrees of freedom which is necessary to denote the $L^2$ projection from virtual element function space to $\mathbb{P}_k(D)$.
In Section 2.4, a semi-norm $|||\cdot|||_{k,D}$ similar as in \cite{Brenner17} is presented. The local estimates for the projections $\Pi_{k,D}^\Delta$ and $\Pi_{k,D}^0$ are obtained. In Section 2.5--2.6, a piecewise $C^1$ polynomial $w$ depending only on the values on $\partial D$ is constructed and the local interpolation error is proved. In section 3, we get the error estimates between $u_h$, its approximations and the exact solution for biharmonic equation. In Section 4, we draw the conclusions. 

\section{Local Virtual Element Spaces in Two Dimensions}
Let $D$ be a polygon in $\mathbb{R}^2$ with diameter $h_D$. 
For a nonnegative integer $k$, $\mathbb{P}_k$ is the space of polynomials of degree $\leq k$ and $\mathbb{P}_{-k} = \{0\},\ k \geq 1.$ The space $\mathbb{P}_k(D)$ is the restriction of $\mathbb{P}_k$ to $D.$ 

The index $(r,s)$ related to the degree of $k\geq 2$ is defined by
$$
r \geq \max\{3,k\},\ s = k-1,\ m=k-4.
$$
A natural choice is $r = \max\{3,k\}$, however, $r$ can be greater. The set of edges of $D$ is denoted by $\mathcal{E}_D$ and $\mathbb{P}_k(e)$ is the restriction of $\mathbb{P}_k$ to $e\in \mathcal{E}_D.$ Then we define
$\mathbb{P}_{r,s}(\partial D)$ as
\begin{eqnarray*}
	\mathbb{P}_{r,s}(\partial D)
	&=& 
	{\bigg \{ }
	v|_e\in \mathbb{P}_r(e),\ 
	\left.\frac{\partial v}{\partial n}\right|_e \in \mathbb{P}_s(e), \forall e\in \mathcal{E}_D,\ \text{and }   
	v,\nabla v \in C(\partial D),\\
	&&\text{ values of $v$,$\nabla v$ at each vertex of $D$ are given degrees of freedom} 
	{\bigg\} }.
\end{eqnarray*}
\subsection{Shape Regularity}\label{Shape-Regularity}
Here the shape regularity assumptions are the same as in \cite{Brenner17}. Let $ {D}$ be the polygon with diameter $h_D$. Assume that 
\begin{equation}\label{assume1}
	| {e}|\geq \rho h_D\quad {\rm for\ any\ edge}\  {e} \ {\rm of}\  {D},\ \rho\in(0,1),
\end{equation}
and 
\begin{equation}\label{assume2}
	{D}\ {\rm is\ star\ shaped\ with\ respect\ to\ a\ disc\ \mathfrak{B}\ with\ radius\ =\ \rho}h_D.
\end{equation}
The center of $\mathfrak{B}$ is the star center of $ {D}.$
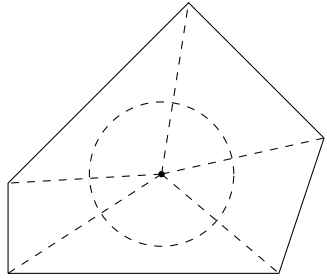
\begin{figure}[H]
	\begin{center}
		\begin{tikzpicture}[scale = 1.2]
			\coordinate (A) at (0,0);
			\coordinate (B) at (3,0);
			\coordinate (C) at (3.5,1.5);
			\coordinate (D) at (2,3);
			\coordinate (E) at (0,1);
			\coordinate (O) at ($1/5*(A)+1/5*(B)+1/5*(C)+1/5*(D)+1/5*(E)$);
			\draw (A)
			--(B)
			--(C)
			--(D)
			--(E)
			--(A);
			\draw[style=dashed](O) circle (0.8);
			\fill [black] (O) circle (1pt);
			\draw[style=dashed](O)--(A);
			\draw[style=dashed](O)--(B);
			\draw[style=dashed](O)--(C);
			\draw[style=dashed](O)--(D);
			\draw[style=dashed](O)--(E);
		\end{tikzpicture}
	\end{center}\caption{A subdomain}\label{fg1}
\end{figure}
The polygon in Figure \ref{fg1} is an example of ${D}$, we denote by $\mathcal{T}_D$ the corresponding triangulation of D.
We will use the notation $A \apprle B$ to represent the inequality $A \leq (constant)B$, where the
positive constant depends only on $k$ and the parameter $\rho$, and it increases with $k$ and $1/\rho$.
The notation $A \approx B$ is equivalent to $A \apprle B$ and $A \apprge B$.
\begin{lemma}\label{bramble}\cite{Bramble70} Bramble-Hilbert Estimates. Conditions \eqref{assume1}-\eqref{assume2} imply that we have the following estimates:
	\begin{equation}
		\inf\limits_{q\in\mathbb{P}_l} |\xi - q|_{H^m(D)} \apprle h_D^{l+1-m}
		|\xi|_{ H^{l+1}(D)}, 
	\end{equation}
for any $\xi\in H^{l+1}(D),\ l = 0,\cdots, k,\ and\ 0\leq m \leq l.$
\end{lemma}
Details can be found in \cite{Brenner07}, Lemma 4.3.8.
\begin{lemma}\label{imbedding}\cite{Adams03} Sobolev Imbedding Theorem 4.12. From \eqref{assume1}-\eqref{assume2}, we have:
	\begin{equation}
		\|\xi\|_{C^j(D)} \apprle \sum\limits_{l=0}^{2+j}h_D^{l-1}
		|\xi|_{ H^{l}(D)}		
		, \ 
		\forall \xi\in H^{2+j}(D),\ j=1,2.
	\end{equation}
\end{lemma}
\begin{lemma}\label{general-poincare}\cite{Necas11} The Generalized Poincar$\acute{\rm e}$ Inequality. From \eqref{assume1}-\eqref{assume2}, suppose $h_D = 1$, we have:
	\begin{equation}
		\|\xi\|_{H^2(D)}^2 \apprle |\xi|_{H^2(D)}^2+\sum\limits_{i=1}^{2}
		\left( \int_{\partial D}\frac{\partial \xi}{\partial x_i}\  {\rm d}x
		\right)^2 +
		\left( \int_{\partial D}{\xi}\  {\rm d}x
		\right)^2 		
		, \quad 
		\forall \xi\in H^{2}(D).
	\end{equation}
\end{lemma}
\begin{proof}
	The proof is similar as the one in \cite{Necas11}, Section 1.1.6.
\end{proof}
\subsection{The Projection $\Pi_{k,D}^\Delta $}
By the generalized poincar$\acute{\rm e}$ inequality from Lemma \ref{general-poincare}, the Sobolev space $H^2(D)$ is a Hilbert space with the inner product $(((\cdot)))$ denoted as:
\begin{eqnarray}\label{inner-p}
	(((\xi,v))) &=& ((\xi,v))_D +\sum\limits_{i=1}^{2}
	\left( \int_{\partial D}\frac{\partial \xi}{\partial x_i}\  {\rm d}x
	\right)\left( \int_{\partial D}\frac{\partial v}{\partial x_i}\  {\rm d}x
	\right) \nonumber\\
	&&+
	\left( \int_{\partial D}{\xi}\  {\rm d}x
	\right)\left( \int_{\partial D}{v}\  {\rm d}x
	\right)	,
\end{eqnarray}
where 
$$
((\xi,v))_D = \int_D 
\sum_{i,j}\frac{\partial^2 \xi}{\partial x_i\partial x_j}
\frac{\partial^2 v}{\partial x_i\partial x_j}\  {\rm d}x,\ i= 1,2; j=1,2,
$$
for any $\xi, v \in H^{2}(D).$

The discrete operator $\Pi_{k,D}^\Delta: H^2(D)\rightarrow \mathbb{P}_k(D)$  is denoted with respect to $(((\cdot)))$ as: 
\begin{equation}\label{ccc}
	(((\Pi_{k,D}^\Delta \xi, q)))_D
	= ((( \xi, q)))_D,\ \forall q\in \mathbb{P}_k(D), 
\end{equation}
in \eqref{ccc}, let $q=1,$ we have \eqref{pikd3}, let $q=x$ then $q=y$, we have \eqref{pikd2},
with the help of \eqref{pikd3} and \eqref{pikd2}, the definition of $\Pi_{k,D}^\Delta$ in \eqref{ccc} can be simplified as \eqref{pikd1}, 
\begin{eqnarray}
	\int_{\partial D} \Pi_{k,D}^\Delta \xi\ {\rm d}s
	&=& \int_{\partial D}   \xi\ {\rm d}s.
	\label{pikd3}\\
	\int_{\partial D} \nabla \Pi_{k,D}^\Delta \xi\  {\rm d}x
	&=& \int_{\partial D} \nabla  \xi\ {\rm d}s,
	\label{pikd2}\\
	((\Pi_{k,D}^\Delta \xi, q))_D
	=& (( &\xi, q))_D, \ \forall \xi\in H^2(D), \ \forall q\in \mathbb{P}_k(D), 
	\label{pikd1}
\end{eqnarray}	
from now on we will use \eqref{pikd3}-\eqref{pikd1} instead of \eqref{ccc}. 

On the domain $D$, with boundary $\partial D$,
we denote by ${\bf n} = (n_1,n_2)$ the outward unit normal vector to $\partial D$,
and by ${\bf t} = (t_1,t_2)$ the unit tangent vector in the counterclockwise
ordering of the boundary. For $u\in H^2(D)$, we define 
$${\bf D}^2 u = (u_{11},u_{22},u_{12},u_{21})=\left(\frac{\partial^2 u}{\partial x_1^2},\frac{\partial^2 u}{\partial x_2^2},\frac{\partial^2 u}{\partial x_1\partial x_2},\frac{\partial^2 u}{\partial x_2\partial x_1}\right).$$
We then denote by
$U_{nn}({\bf D}^2 u) := \sum\limits_{i,j}u_{ij} n_i n_j$ the normal bending moment, by $U_{nt}({\bf D}^2 u) := \sum\limits_{i,j}u_{ij} n_i t_j$
the twisting moment, and by $Q_{n}({\bf D}^2 u) := \sum\limits_{i,j}\frac{\partial u_{ij}}{\partial x_i} n_j$ the normal shear force, and $U_{\Delta}({\bf D}^2 u) = \Delta^2 u$.

After applying integration by parts twice, we have
\begin{eqnarray}\label{eq_bp1}
	((u,v))_D&=&\int_D U_{\Delta}({\bf D}^2 u) v \ {\rm d}x
	+\int_{\partial D} U_{nn}({\bf D}^2 u)\frac{\partial v}{\partial n}\ {\rm d}s\nonumber \\
	&&-\int_{\partial D} \left(Q_{n}({\bf D}^2 u)+\frac{\partial U_{nt}({\bf D}^2 u)}{\partial t}\right)v\ {\rm d}s.
\end{eqnarray}

\subsection{Local VEM Space $\mathcal{Q}_k(D)$}
Then, for $k\geq 2,$ the local VEM space $\mathcal{Q}^k(D)\in H^2(D)$ is defined as: $v\in H^2(D)$ belongs to $\mathcal{Q}^k(D)$ if and only if (i) $v|_{\partial D}$ and trace of $\frac{\partial v}{\partial n}$ on $\partial D$ belongs to $\mathbb{P}_{r,s}(\partial D),$ 
then (ii) there exists a polynomial $q_v(= \Delta^2v) \in \mathbb{P}_k(D)$ such that
\begin{equation}\label{var-form}
	((v,w))_D= (q_v, w),\quad \forall w\in H^2_0(D),
\end{equation}
and (iii)
\begin{equation}\label{equal-parts}
	\Pi_{k,D}^0v-\Pi_{k,D}^\Delta v \in \mathbb{P}_{k-4}(D),
\end{equation}
where $\Pi_{k,D}^0$ is the projection from $L^2(D)$ onto $\mathbb{P}_k(D)$.
\begin{remark}
	It's obvious that $\mathbb{P}_k(D)$ is a subspace of $\mathcal{Q}^k(D).$ From \eqref{equal-parts}, we have
	$\Pi_{k,D}^0v=\Pi_{k,D}^\Delta v,$ $k=2,3$. \\
	The choice (ii) can be replaced by $q_v(= \Delta^2v) \in \mathbb{P}_{k-2}(D)$ in \cite{Chinosi16},  also  Lemma \ref{equivalence_b}--Lemma \ref{interpolation_error}, and Corollary \ref{corollary_1}--Corollary \ref{corollary_2} are valid. The reason we chose $q_v\in \mathbb{P}_{k}(D)$ is that it helps to get the same error estimate with less smooth right hand side $f$.	For $k=2,3$, we only require $f\in L^2(\Omega)$, while in \cite{Chinosi16}, $f\in H^1(\Omega)$ for $k=2$, and $f\in H^2(\Omega)$ for $k=3$. 
\end{remark}

For completeness, we recall the definition for degrees of freedom in \cite{Brezzi13}, employ the following notation: for $i$ a nonnegative integer, $e$ an edge with midpoint $x_e$, length $h_e$, the set of $i+1$ normalized monomials is denoted by $\mathcal{M}_i^e$
$$
\mathcal{M}_i^e :=\left\{1,\frac{x-x_e}{h_e},\left(\frac{x-x_e}{h_e}\right)^2,\cdots, \left(\frac{x-x_e}{h_e}\right)^i\right\}.
$$
And for domain $D$ with diameter $h_D$ and barycenter ${\bf x}_D$, the set of $(i+1)(i+2)/2$ normalized monomials is defined by $\mathcal{M}_i^D$
$$
\mathcal{M}_i^D :=\left\{ \left(\frac{x-{\bf x}_D}{h_D}\right)^\alpha,\quad |\alpha|\leq i\right\},
$$
where $\alpha$ is a nonnegative multiindex $\alpha = (\alpha_1,\alpha_2)$, $|\alpha|=\alpha_1+\alpha_2$ and ${\bf x}^\alpha=x_1^{\alpha_1} x_2^{\alpha_2}$. In $D$ the degrees of freedom are denoted as:
\begin{equationate}
	\item{$\bullet$ The values $v$ at each vertex of $D$.} \label{bul1}
	\item{$\bullet$ The values $\nabla v$ at each vertex of $D$.}
	\item{$\bullet$ For $r>3$, the moments $\frac{1}{h_e}\int_e q(s)v\ {\rm d}s,\ \forall q\in \mathcal{M}_{r-4}^e,\ \forall e\in \partial D$.} 
	\item{$\bullet$ For $s>1$, the moments $\frac{1}{h_e}\int_e q(s)\frac{\partial v}{\partial n}\ {\rm d}s,\ \forall q\in \mathcal{M}_{s-2}^e,\ \forall e\in \partial D$.} 
	\item{$\bullet$ For $m\geq 0$, the moments $\frac{1}{|D|}\int_D q(x)v(x)\ {\rm d}x,\ \forall q\in \mathcal{M}_{m}^D$.} \label{bul5}
\end{equationate}
\begin{lemma}\label{lemma1}
	Given any $g\in \mathbb{P}_{r,s}({\partial{D}})$ and $f\in \mathbb{P}_k(D)$, there exists a unique function $v\in H^2(D)$ such that (i) $v=g,$ $\frac{\partial v}{\partial n}=\frac{\partial g}{\partial n}$ on $\partial D$ and (ii)
	$$
	\int_D{\bf D}^2 v \cdot {\bf D}^2 w\ {\rm d}x= \int_{D} fw\ {\rm d}x,\quad \forall w\in H^2_0(D). 
	$$
\end{lemma}	
\begin{proof}
	Similar as in \cite{Brenner17}, let $\tilde{g} \in H^2(D)$ be a $C^1$, $\mathbb{P}_k$ piece-wise polynomial constructed in Section \ref{case1}, such that $v=g,$ $\frac{\partial v}{\partial n}=\frac{\partial g}{\partial n}$ on $\partial D$. Then the unique $v\in H^2(D)$ is given by $\phi+\tilde{g}$, where $\phi\in H^2_0(D)$ is defined by
	$$
	\int_D{\bf D}^2 \phi \cdot {\bf D}^2 w\ {\rm d}x= \int_{D} fw\  {\rm d}x
	-\int_D{\bf D}^2 \tilde{g} \cdot {\bf D}^2 w\ {\rm d}x,
	\quad \forall w\in H^2_0(D). 
	$$
	The proof is completed.
\end{proof}
\begin{lemma}
	We have (i) {\rm dim} $\mathcal{Q}^k(D)$ =  {\rm dim} $\mathbb{P}_{r,s}(\partial D)$ {\rm + dim} $\mathbb{P}_{m}(D)$, and (ii) $v\in \mathcal{Q}^k(D)$ is uniquely determined by $v|_{\partial D}, \left.\frac{\partial v}{\partial n}\right|_{\partial D}$ and $\Pi_{k-4,D}^0v$. 
\end{lemma}
\begin{proof} 
	Following \cite{Brenner17} and \cite{Brezzi13}, 
	let 
	$${\mathcal{Q}}^k_D:=\left\{v\in H^2(D), v|_{\partial D}, \left.\frac{\partial v}{\partial n}\right|_{\partial D} \in \mathbb{P}_{r,s}(\partial D)\text{ and } \Delta^2 v\in \mathbb{P}_k(D)  \right\}$$. 
	The linear map 
	$v\mapsto (v|_{\partial D}, \left.\frac{\partial v}{\partial n}\right|_{\partial D}, \Delta^2 v)$ 
	from ${\mathcal{Q}}^k_D$ to $\mathbb{P}_{r,s}(\partial D)\times \mathbb{P}_k(D)$ is an isomorphism by Lemma \ref{lemma1}.  
	
	The linear map 
	$v\mapsto (v|_{\partial D}, \left.\frac{\partial v}{\partial n}\right|_{\partial D}, \Pi_{k-4,D}^0 v +(\Pi_{k,D}^0 -\Pi_{k-4,D}^0)(v-\Pi_{k,D}^\Delta v) )$ 
	from ${\mathcal Q^k}(D)$ to $\mathbb{P}_{r,s}(\partial D)\times \mathbb{P}_k(D)$ is also an isomorphism. Suppose $v\in$ null space, then $\Pi_{k-4,D}^0v = 0$,
	$$
	v|_{\partial D} = 0\ {\rm and}\
	\left.\frac{\partial v}{\partial n}\right|_{\partial D} =0.
	$$ 
	With \eqref{pikd1}-\eqref{pikd3} and \eqref{eq_bp1}, we have
	$\Pi_{k,D}^\Delta v = 0$, so that by \eqref{equal-parts}, 
	$$
	0=\Pi_{k-4,D}^0v=\Pi_{k,D}^0v \in \mathbb{P}_{k-4}(D).
	$$ 
	In \eqref{var-form}, let $w = v\in \mathcal{Q}^k(D)$, then we have
	$$
	|v|^2_{H^2(D)} = 0 \Rightarrow v=0.
	$$
	This completes the proof.
\end{proof}
\begin{lemma}\label{discrete_estimates}
	Discrete Estimates. From Conditions \eqref{assume1}-\eqref{assume2}, and the equivalence  of norms on finite dimensional vector spaces, for any $u\in \mathbb{P}_k$, we have the following estimates:
	$$
	\|{\bf D}^2 u\|_{L^2(D)}
	\apprle
	\|u\|_{L^2(D)}\quad \text{and}\quad \left\|\frac{\partial u}{\partial t}\right\|_{L^2(e)}\apprle h_e^{-1/2}\|u\|_{L^2(e)},
	$$
\begin{eqnarray*}
h_D^2\|U_{\Delta}({\bf D}^2 u)\|_{L^2(D)} 
&+&\|U_{nn}({\bf D}^2 u)\|_{L^2(D)}\\ 
&+&h_D^{3/2}\left\|Q_{n}({\bf D}^2 u)+\frac{\partial U_{nt}({\bf D}^2 u)}{\partial t}\right\|_{L^2(\partial D)} 
\apprle 
\|{\bf D}^2 u\|_{L^2(D)}.
\end{eqnarray*}	
\end{lemma}
\subsection{Estimates of $|||\cdot|||_{k,D}$}\label{sec1.4}
The semi-norm $|||\cdot|||_{k,D}$ for $\xi\in H^2(D)$ is defined by
\begin{eqnarray}\label{eq_v-norm}
	|||\xi|||_{k,D}^2
	&=& 
	\|\Pi^0_{k-4,D}\xi\|_{L^2(D)}^2
	+h_D\sum_{e\in\mathcal{E}_D}\|\Pi^0_{r,e}\xi\|_{L^2(e)}^2\nonumber \\
	&&+h_D^3\sum_{e\in\mathcal{E}_D, i=1,2}\left\|\Pi^0_{r-1,e}\frac{\partial \xi}{\partial x_i}\right\|_{L^2(e)}^2 .
\end{eqnarray}
There is an obvious stability estimate from \eqref{pikd1}
\begin{equation}\label{obvious_ineq}
	|\Pi_{k,D}^\Delta\xi|_{H^2(D)}\leq |\xi|_{H^2(D)}, \  \forall \xi\in H^2(D).
\end{equation}
We define the kernel of operator $\Pi_{k,D}^\Delta$ as: 
$$
{\rm Ker} \Pi_{k,D}^\Delta := \{v\in\mathcal{Q}^k(D): \Pi_{k,D}^\Delta v = 0\}.
$$
And for any $ v \in\mathcal{Q}^k(D)$, we have
\begin{eqnarray*}
	\sum_{e\in\mathcal{E}_D}\|\Pi^0_{r,e}v\|_{L^2(e)}^2
	&=& 
	\|v\|_{L^2(\partial D)}^2,\\
	\sum_{e\in\mathcal{E}_D}\left\|\Pi^0_{r-1,e}\frac{\partial v}{\partial x_i}\right\|_{L^2(e)}^2 
	&=&
	\left\|\frac{\partial v}{\partial x_i}\right\|_{L^2(\partial D)}^2.
\end{eqnarray*}
\begin{lemma}\label{new-equi}
	For any $v \in \mathcal{Q}^k(D)$, we have the equivalence of norms: 
	\begin{eqnarray*}
		|||v|||_{k,D}
		&=&
		\|\Pi^0_{k-4,D}\xi\|_{L^2(D)}^2
		+h_D\|v\|_{L^2(\partial D)}^2
		+h_D^3\sum_{i=1,2}\left\|\frac{\partial v}{\partial x_i}\right\|_{L^2(\partial D)}^2
		\\
		&\approx& \|\Pi^0_{k-4,D}\xi\|_{L^2(D)}^2
		+h_D\|v\|_{L^2(\partial D)}^2
		+h_D^3\left\|\frac{\partial v}{\partial n}\right\|_{L^2(\partial D)}^2 .
	\end{eqnarray*}
\end{lemma}
\begin{proof}
	Suppose $h_D =1$, by the discrete estimates from Lemma \ref{discrete_estimates}, we have
	$$
	\left\|\frac{\partial v}{\partial t}\right\|_{L^2(\partial D)}
	\apprle \|v\|_{L^2(\partial D)},
	$$
	and  
	\begin{eqnarray*}
		\left\|\frac{\partial v}{\partial t}\right\|^2_{L^2(\partial D)}+\left\|\frac{\partial v}{\partial n}\right\|^2_{L^2(\partial D)}
		\approx \sum_{i=1,2}\left\|\frac{\partial v}{\partial x_i}\right\|_{L^2(\partial D)}^2,
	\end{eqnarray*}
	so that the equivalence is obtained.
\end{proof}
\begin{lemma}\label{pq_ineq}
	For any $p\in \mathbb{P}_{k-4}$, $k\geq 2$, there exists $q  \in \mathbb{P}_{k}$, such that $\Delta^2 q = p$ and
	$$
	\|q\|_{L^2(D)}\apprle \|p\|_{L^2(D)}.
	$$
\end{lemma}
\begin{proof}
	From \cite{Brenner17}, we know that $\Delta$ maps $\mathbb{P}_k$ to $\mathbb{P}_{k-2}$, so that $\Delta^2=\Delta\Delta$  maps $\mathbb{P}_k$ to $\mathbb{P}_{k-4}$. Then there exists an operator $(\Delta^2)^{\dagger}: \mathbb{P}_{k-4}\rightarrow \mathbb{P}_{k}$ such that $\Delta^2 (\Delta^2)^{\dagger}$ is the identity operator on $\mathbb{P}_{k-4}$.
	We define the norm of $p$ as
	$$
	\|p\|_{(\Delta^2)^{\dagger}}
	:= 
	\inf_{(\Delta^2)^{\dagger}p\in \mathbb{P}_{k}}
	\|(\Delta^2)^{\dagger}p\|_{L^2(D)}.
	$$
	Since we have
	$$
	q = \sum\limits_{i,j=0;i+j\leq k} c_{ij}x_1^i x_2^j\ \text{  and  }\ \Delta^2 q = p,
	$$
	the minimization problem
	$$
	\|p\|_{(\Delta^2)^{\dagger}} = \inf_{c_{i,j}}
	\|q\|_{L^2(D)},
	$$
	is solvable.
	So, there exists $q=(\Delta^2)^{\dagger}p$ such that
	$$
	\|q\|_{L^2(D)} = \|p\|_{(\Delta^2)^{\dagger}}.
	$$
	By the equivalence of norms on finite dimensional vector space, we have
	$$
	\|p\|_{(\Delta^2)^{\dagger}}\apprle \|p\|_{L^2(D)},
	$$
	then the result is obtained.
\end{proof}
\begin{lemma}\label{equivalence_b}
	For any $v\in {\rm Ker} \Pi_{k,D}^\Delta$, we have
	$$
	|||v|||^2_{k,D}\approx 
	h_D\left\|v\right\|_{L^2(\partial D)}^2
	+
	h_D^3\left\|\frac{\partial v}{\partial n}\right\|_{L^2(\partial D)}^2.
	$$	
\end{lemma}
\begin{proof}
	Suppose $h_D = 1$. For $k<4$, $\Pi_{k-4,D}^0 v =0,$ the equivalence is trivial.
	For $k\geq 4,$ let $v\in {\rm Ker} \Pi_{k,D}^\Delta$,  by \eqref{pikd1}, \eqref{eq_bp1}, and using the same $p$, $q$ in Lemma \ref{pq_ineq}, we have
	\begin{equation}\label{pqker}
		\int_D  v ({\Delta}^2q)\ {\rm d}x
		=
		\int_{\partial D} \left(Q_{n}({{\bf D}^2 q})+\frac{\partial U_{nt}({{\bf D}^2 q})}{\partial t}\right)v\ {\rm d}s
		-\int_{\partial D} U_{nn}({{\bf D}^2 q})\frac{\partial v}{\partial n}\ {\rm d}s.
	\end{equation}
	By lemma \ref{discrete_estimates}, we have
	$$
	\left\|Q_{n}({{\bf D}^2 q})+\frac{\partial U_{nt}({{\bf D}^2 q})}{\partial t}\right\|_{L^2(\partial D)}
	\apprle \|{{\bf D}^2 q}\|_{L^2(D)} 
	\apprle \|{q}\|_{L^2(D)} 
	$$
	and $\left\|U_{nn}({{\bf D}^2 q})\right\|_{L^2(\partial D)}\apprle \|{q}\|_{L^2(D)}$.
	Then, by Lemma  \ref{pq_ineq},
	$$
	\left|\int_D  v p\ {\rm d}x\right|
	\apprle 
	\left( 
	\left\|v\right\|_{L^2(\partial D)}
	+
	\left\|\frac{\partial v}{\partial n}\right\|_{L^2(\partial D)}
	\right)\|{q}\|_{L^2(D)},
	$$
	so that
\begin{eqnarray*}
	\|\Pi_{k-4,D}^0 v\|_{L^2(D)} 
	&=&
	\max\limits_{p\in \mathbb{P}_{k-4}}\left|\int_D  (\Pi_{k-4,D}^0 v)  (p/\|{p}\|_{L^2(D)})\ {\rm d}x\right|\\
	&\apprle&  
	\left\|v\right\|_{L^2(\partial D)}
	+
	\left\|\frac{\partial v}{\partial n}\right\|_{L^2(\partial D)},
\end{eqnarray*}
	which means
	$$
	\|\Pi_{k-4,D}^0 v\|_{L^2(D)}^2 
	\apprle  
	\left\|v\right\|_{L^2(\partial D)}^2
	+
	\left\|\frac{\partial v}{\partial n}\right\|_{L^2(\partial D)}^2,
	$$
	with Lemma \ref{new-equi}, we get the result.
\end{proof}
\begin{remark}\label{remark_t}
	Same as in \cite{Brenner17}, we have
	$$
	|||v|||^2_{k,D}\approx 
	h_D^3\left\|\frac{\partial v}{\partial t}\right\|_{L^2(\partial D)}^2
	+
	h_D^3\left\|\frac{\partial v}{\partial n}\right\|_{L^2(\partial D)}^2,\quad \forall v \in  {\rm Ker} \Pi_{k,D}^\Delta,
	$$
	where $\partial/\partial t$ denotes a tangential derivative along $\partial D$.
\end{remark}
There is also a stability estimate for $\Pi_{k,D}^\Delta\xi$  in $H^1(D)$ norm in terms of the semi-norm $|||\cdot|||_{k,D}$.
\subsection{Estimates of $\Pi_{k,D}^\Delta$ and $\Pi_{k,D}^0$}
\begin{lemma}\label{leq-semi-norm} We have
	\begin{equation}
		\|\Pi_{k,D}^\Delta\xi\|_{L^2(D)}+h_D|\Pi_{k,D}^\Delta\xi|_{H^1(D)}+h_D^2|\Pi_{k,D}^\Delta\xi|_{H^2(D)}  \apprle |||\xi|||_{k,D},\  \forall \xi\in H^2(D).
	\end{equation}
\end{lemma}
\begin{proof}
	Suppose $u=\Pi_{k,D}^\Delta\xi$ and $h_D = 1$, by \eqref{pikd1}, \eqref{eq_bp1}, we have 
	\begin{eqnarray}\label{eq_in}
		|\Pi_{k,D}^\Delta\xi|_{H^2(D)}^2 
		&=& ((\Pi_{k,D}^\Delta\xi,\xi))_D\nonumber \\
		&=&
		\int_D U_{\Delta}({\bf D}^2 (\Pi_{k,D}^\Delta\xi)) \xi \ {\rm d}x
		+\int_{\partial D} U_{nn}({\bf D}^2 (\Pi_{k,D}^\Delta\xi))\frac{\partial \xi}{\partial n}\ {\rm d}s\nonumber \\
		&&-\int_{\partial D} \left(Q_{n}({\bf D}^2 u(\Pi_{k,D}^\Delta\xi))+\frac{\partial U_{nt}({\bf D}^2 (\Pi_{k,D}^\Delta\xi))}{\partial t}\right)\xi\ {\rm d}s, 
	\end{eqnarray}
	then, by Cauchy-Schwarz inequality, Lemma \ref{discrete_estimates} and \eqref{eq_v-norm}, we have
	\begin{eqnarray}\label{eq_in2}
		|\Pi_{k,D}^\Delta\xi|_{H^2(D)}^2 
		\apprle
		|\Pi_{k,D}^\Delta\xi|_{H^2(D)} |||\xi|||_{k,D}
		\apprle
		|||\xi|||_{k,D}^2.
	\end{eqnarray}
	By \eqref{pikd3}-\eqref{pikd2}, and Lemma \ref{general-poincare}, we have  
	\begin{equation}\label{norm_P}
		\|\Pi_{k,D}^\Delta\xi\|^2_{H^2(D)} \apprle |\Pi_{k,D}^\Delta\xi|_{H^2(D)}^2+\sum_{i=1}^{2}\left(\int_{\partial D}\frac{\partial \Pi_{k,D}^\Delta\xi }{\partial x_i}\ {\rm d}s\right)^2+\left(\int_{\partial D} \Pi_{k,D}^\Delta \xi\ {\rm d}s\right)^2
	\end{equation}
	also we have 
	$$
	\int_{\partial D}\frac{\partial \Pi_{k,D}^\Delta\xi }{\partial x_i}\ {\rm d}s
	=\int_{\partial D}\frac{\partial \xi }{\partial x_i}\ {\rm d}s
	$$
	\begin{equation}\label{bound_xi}
		\left|\int_{\partial D}\frac{\partial \xi }{\partial x_i}\ {\rm d}s\right|
		\leq
		\sum_{e\in\mathcal{E}_D}\left|\int_{e}{\Pi_{0,e}^0\frac{\partial \xi }{\partial x_i} }\ {\rm d}s\right|
		\apprle
		\left(\sum_{e\in\mathcal{E}_D}
		\left\|\Pi^0_{r-1,e}\frac{\partial \xi }{\partial x_i}
		\right\|_{L^2(e)}^2\right)^{1/2}
	\end{equation}
	Similarly, 
	\begin{equation}\label{bound_xi1}
		\left(\int_{\partial D} \Pi_{k,D}^\Delta \xi\ {\rm d}s\right)^2
		\apprle
		\sum_{e\in\mathcal{E}_D}\|\Pi^0_{r,e}\xi\|_{L^2(e)}^2.
	\end{equation}
	From \eqref{eq_in}-\eqref{bound_xi1}, the following inequality is valid
	\begin{equation}
		\|\Pi_{k,D}^\Delta\xi\|_{H^2(D)} \apprle|||\xi|||.
	\end{equation}
The proof is completed.
\end{proof}
\begin{lemma}\label{friedrichs}
	For any $\xi\in H^2(D)$, we have
	\begin{equation}\label{upper_bound1}
		|||\xi |||_{k,D}
		\apprle 
		\|\xi\|_{L^2(D)}
		+
		h_D|\xi|_{H^1(D)}
		+
		h_D^2|\xi|_{H^2(D)},
	\end{equation}
	and there exists $\bar{\xi}\in \mathbb{P}_1$, such that
	\begin{equation}\label{upper_bound2}
		|||\xi - \bar{\xi} |||_{k,D}
		\apprle 
		h_D^2|\xi|_{H^2(D)},
	\end{equation}	
	where 
	$$
	\int_{\partial D}\nabla\bar{\xi} \ {\rm d}x= \int_{\partial D}\nabla\xi \ {\rm d}x,
	$$
	$$
	\int_{\partial D}\bar{\xi} \ {\rm d}x= \int_{\partial D}\xi \ {\rm d}x.
	$$
\end{lemma}
\begin{proof}
	Assume $h_D$ = 1, then by trace theorem
	$$|||\xi |||_{k,D}
	\apprle 
	\|\xi\|_{L^2(D)}
	+\|\xi\|_{L^2(\partial D)}
	+\sum_{i=1,2}\left\|\frac{\partial \xi}{\partial x_i}\right\|_{L^2(\partial D)}
	\apprle 
	\|\xi\|_{L^2(D)}
	+
	|\xi|_{H^1(D)}
	+
	|\xi|_{H^2(D)}.
	$$
	So that we have \eqref{upper_bound1}, and by Lemma \ref{general-poincare}
	$$|||\xi -\bar{\xi} |||_{k,D}^2
	\apprle 
	\|\xi-\bar{\xi}\|_{H^2(D)}^2
	\apprle
	|\xi|_{H^2(D)}^2+\sum\limits_{i=1}^{2}
	\left( \int_{\partial D}\frac{\partial (\xi -\bar{\xi})}{\partial x_i}\ {\rm d}x
	\right)^2 +
	\left( \int_{\partial D}{\xi -\bar{\xi}}\ {\rm d}x
	\right)^2 		
	$$
	with the definition of $\bar{\xi}$, we arrived at
	$$|||\xi -\bar\xi |||_{k,D}
	\apprle 
	|\xi|_{H^2(D)}.
	$$
The proof is completed.
\end{proof}

\begin{corollary}\label{corollary_1}
	We have
	\begin{eqnarray}
		\|\xi - \Pi_{k,D}^\Delta\xi\|_{L^2(D)}
		&\apprle&
		h_D^{l+1}|\xi|_{H^{l+1}(D)},\quad \forall \xi\in H^{l+1}(D),\ 1\leq l\leq k,\label{l2_error}
		\\
		|\xi - \Pi_{k,D}^\Delta\xi|_{H^{1}(D)}
		&\apprle&
		h_D^{l}|\xi|_{H^{l+1}(D)},\quad \forall \xi\in H^{l+1}(D),\ 1\leq l\leq k,\label{h1_error}
		\\
		|\xi - \Pi_{k,D}^\Delta\xi|_{H^{2}(D)}
		&\apprle&
		h_D^{l-1}|\xi|_{H^{l+1}(D)},\quad \forall \xi\in H^{l+1}(D),\ 1\leq l\leq k.\label{h2_error}
	\end{eqnarray}
\end{corollary}
\begin{proof}
	From Lemma \ref{bramble}, Lemma \ref{leq-semi-norm}, and Lemma \ref{friedrichs}, for any $q\in \mathbb{P}_l, \ \xi\in H^{l+1}(D)$ we have
	$$
	\|\xi - \Pi_{k,D}^\Delta\xi\|_{L^2(D)}\apprle\|\xi - q\|_{L^2(D)}+\|\Pi_{k,D}^\Delta (q - \xi)\|_{L^2(D)}\apprle h_D^{l+1}|\xi|_{H^{l+1}(D)},
	$$
	so as the $H^1$ and $H^2$ error estimates.
\end{proof}
For the $L^2$ operator $\Pi_{k,D}^0$,  from \cite{Brenner17}, we have
\begin{equation}\label{pikd0l2}
	\|\Pi_{k,D}^0\xi\|_{L^2(D)}\leq \|\xi\|_{L^2(D)},
	\quad 
	|\Pi_{k,D}^0\xi|_{H^1(D)}\apprle |\xi|_{H^1(D)},\quad 
	\forall \xi\in H^1(D),
\end{equation}
and
\begin{equation}\label{pikd0h1}
	|\xi - \Pi_{k,D}^0\xi|_{H^1(D)}\apprle h_D^{l}|\xi|_{H^{l+1}(D)},\quad \forall \xi\in H^{l+1}(D),\ 1\leq l\leq k.
\end{equation}
\begin{lemma}\label{pikd0-bound}
	The estimates of $\Pi_{k,D}^0$ satisfy
	\begin{eqnarray}
		|\Pi_{k,D}^0\xi|_{H^{2}(D)}
		&\apprle&
		|\xi|_{H^{2}(D)},\quad \forall \xi\in H^{2}(D),\label{pikd0h2}
		\\
		|\xi - \Pi_{k,D}^0\xi|_{H^{2}(D)}
		&\apprle&
		h_D^{l-1}|\xi|_{H^{l+1}(D)},\quad \forall \xi\in H^{l+1}(D),\ 1\leq l\leq k.\label{pikd0h2_error}
	\end{eqnarray}
\end{lemma}
\begin{proof}
	Here, suppose $\xi\in H^2(D)$, then by \eqref{obvious_ineq}, \eqref{h1_error}, \eqref{pikd0h1} and the inverse inequality, we have
	\begin{eqnarray*}
		|\Pi_{k,D}^0\xi|_{H^2(D)} 
		&\apprle &
		|\Pi_{k,D}^0\xi-\Pi_{k,D}^\Delta\xi|_{H^2(D)}
		+|\Pi_{k,D}^\Delta \xi|_{H^2(D)}
		\\
		&\apprle &
		h_D^{-1}|\Pi_{k,D}^0\xi-\xi+\xi-\Pi_{k,D}^\Delta\xi|_{H^1(D)}
		+|\xi|_{H^2(D)}
		\\
		&\apprle &
		|\xi|_{H^2(D)}
	\end{eqnarray*}
which completes the proof.
\end{proof}
\subsection{An Inverse Inequality} 
\begin{lemma}\label{minimum-energy}
	The following inequality is valid
	\begin{equation}
		|v|_{H^2(D)} \leq |\xi|_{H^2(D)},
	\end{equation}
	for any $v\in \mathcal{Q}^k(D)$ and $\xi\in H^2(D)$ such that $\xi=v$ on $\partial D$, $\frac{\partial \xi}{\partial n}=\frac{\partial v}{\partial n}$ on $\partial D$, and
	$\Pi_{k,D}^0(\xi -v) = 0.$ 
\end{lemma}
\begin{proof}
	In (11), let $u$ be $v$, $v$ be $\xi-v$, the second and third terms at right hand side can be eliminated.  Also $\xi$ satisfies $\Pi_{k,D}^0(\xi-v)=0$,  and
	$U_{\Delta}({\bf D}^2v) \in \mathbb{P}_k(D)$ so that $((v,\xi-v))_D = 0$.
	Then by \eqref{pikd1} and \eqref{eq_bp1}, we have
	$$
	((v,\xi-v))_D = ((v,\xi))_D-|v|^2_{H^2(D)}=0
	$$
	And hence,
	$$
	|\xi|^2_{H^2(D)} = |\xi-v|^2_{H^2(D)} +|v|^2_{H^2(D)},
	$$	
	which means  
	\begin{equation*}
		|v|_{H^2(D)} \leq |\xi|_{H^2(D)}. 
	\end{equation*}
So the proof is completed.
\end{proof}
Next, we will consider the relation between
$|v|_{H^2(D)}$ and $|||v|||_{k,D}$, $\forall v \in \mathcal{Q}^k(D)$.
\subsubsection{Construction of $w$}\label{case1}
The degree of freedom of $v\in \mathcal{Q}^k(D)$ is defined in \cite{Brezzi13}, from (4.7)--(4.11).
For $k\geq 2,$ we will employ the triangulation $\mathcal{T}_D$ to define a piecewise polynomial $w$ which has the same boundary conditions as $v$. On each internal triangle, we employ a $\mathbb{P}_{r}$ macroelement, which is defined in \cite{Douglas79}, Section 1. Suppose $k=2,3$, in Figure \ref{fg-k2k3},  on each internal triangle, the function is defined by $\mathbb{P}_3$ macroelement as in Figure \ref{fg-p3}. All degrees of freedom within $D$ are 0.
\begin{figure}[H]
	\begin{center}
		\begin{tikzpicture}[scale = 1.2]
			\coordinate (A) at (0,0);
			\coordinate (B) at (3,0);
			\coordinate (C) at (3.5,1.5);
			\coordinate (D) at (2,3);
			\coordinate (E) at (0,1);
			\coordinate (O) at ($1/5*(A)+1/5*(B)+1/5*(C)+1/5*(D)+1/5*(E)$);
			\draw (A)
			--(B)
			--(C)
			--(D)
			--(E)
			--(A);
			\draw[style=dashed](O)--(A);
			\draw[style=dashed](O)--(B);
			\draw[style=dashed](O)--(C);
			\draw[style=dashed](O)--(D);
			\draw[style=dashed](O)--(E);
			
			\fill [black] (A) circle (1pt);
			\draw [black] (A) circle (2.5pt);
			\fill [black] (B) circle (1pt);
			\draw [black] (B) circle (2.5pt);
			\fill [black] (C) circle (1pt);
			\draw [black] (C) circle (2.5pt);
			\fill [black] (D) circle (1pt);
			\draw [black] (D) circle (2.5pt);
			\fill [black] (E) circle (1pt);
			\draw [black] (E) circle (2.5pt);			
			\coordinate (A) at (0+5,0);
			\coordinate (B) at (3+5,0);
			\coordinate (C) at (3.5+5,1.5);
			\coordinate (D) at (2+5,3);
			\coordinate (E) at (0+5,1);
			\coordinate (O) at ($1/5*(A)+1/5*(B)+1/5*(C)+1/5*(D)+1/5*(E)$);
			\coordinate (AB) at ($(A)!1/2!(B)$);
			\coordinate (BC) at ($(B)!1/2!(C)$);
			\coordinate (CD) at ($(C)!1/2!(D)$);
			\coordinate (DE) at ($(D)!1/2!(E)$);
			\coordinate (EA) at ($(E)!1/2!(A)$);
			\draw (A)
			--(B)
			--(C)
			--(D)
			--(E)
			--(A);
			\draw[style=dashed](O)--(A);
			\draw[style=dashed](O)--(B);
			\draw[style=dashed](O)--(C);
			\draw[style=dashed](O)--(D);
			\draw[style=dashed](O)--(E);
			\fill [black] (A) circle (1pt);
			\draw [black] (A) circle (2.5pt);
			
			\fill [black] (B) circle (1pt);
			\draw [black] (B) circle (2.5pt);
			
			\fill [black] (C) circle (1pt);
			\draw [black] (C) circle (2.5pt);
			
			\fill [black] (D) circle (1pt);
			\draw [black] (D) circle (2.5pt);
			
			\fill [black] (E) circle (1pt);
			\draw [black] (E) circle (2.5pt);
			\draw  ($(AB)$) -- ($(AB)!0.2cm!90:(A)$);
			\draw  ($(AB)$) -- ($(AB)!0.2cm!-90:(A)$);	
			
			\draw  ($(BC)$) -- ($(BC)!0.2cm!90:(B)$);
			\draw  ($(BC)$) -- ($(BC)!0.2cm!-90:(B)$);	
			
			\draw  ($(CD)$) -- ($(CD)!0.2cm!90:(C)$);
			\draw  ($(CD)$) -- ($(CD)!0.2cm!-90:(C)$);	
			
			\draw  ($(DE)$) -- ($(DE)!0.2cm!90:(D)$);
			\draw  ($(DE)$) -- ($(DE)!0.2cm!-90:(D)$);	
			
			\draw  ($(EA)$) -- ($(EA)!0.2cm!90:(E)$);
			\draw ($(EA)$)--($(EA)!0.2cm!-90:(E)$);					
		\end{tikzpicture}		
	\end{center}\caption{Local d.o.f. for the lowest-order element: $k = 2,\ (r,s,m)=(3,1,-2)$ (left), and next to the lowest element:
		$k = 3,\ (r,s,m)=(3,2,-1)$ (right).}\label{fg-k2k3}
\end{figure}
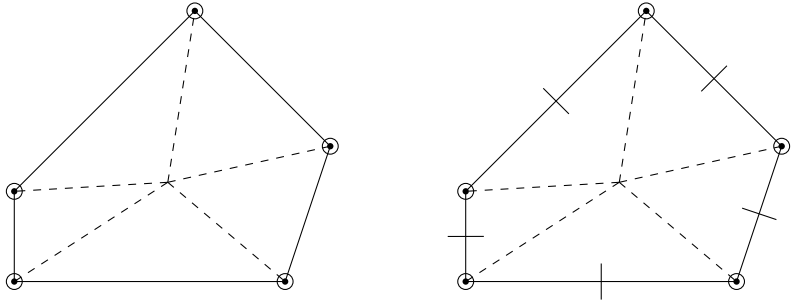
\begin{figure}[H]
	\begin{center}
		\begin{tikzpicture}[scale = 1.3]
			\coordinate (D) at (0,0);
			\coordinate (E) at (3,0);
			\coordinate (O) at (1.5,2.5);
			\coordinate (DE) at ($(D)!1/2!(E)$);
			\coordinate (OD) at ($1/2*(D)+1/2*(O)$);
			\coordinate (OE) at ($(O)!1/2!(E)$);
			\coordinate (center) at ($1/3*(D)+1/3*(O)+1/3*(E)$);
			\draw (D)--(E);
			\draw (E)--(O);
			\draw (O)--(D);
			\draw(center)--(D);
			\draw(center)--(E);
			\draw(center)--(O);
			\fill [black] (D) circle (1pt);
			\draw [black] (D) circle (2.5pt);
			
			\fill [black] (E) circle (1pt);
			\draw [black] (E) circle (2.5pt);
			
			\fill [black] (O) circle (1pt);
			\draw [black] (O) circle (2.5pt);
			
			\draw  ($(OD)$) -- ($(OD)!0.2cm!90:(O)$);
			\draw  ($(OD)$) -- ($(OD)!0.2cm!-90:(O)$);
			
			\draw  ($(DE)$) -- ($(DE)!0.2cm!90:(D)$);
			\draw  ($(DE)$) -- ($(DE)!0.2cm!-90:(D)$);
			
			\draw  ($(OE)$) -- ($(OE)!0.2cm!90:(O)$);
			\draw  ($(OE)$) -- ($(OE)!0.2cm!-90:(O)$);
		\end{tikzpicture}
	\end{center}\caption{$\mathbb{P}_3$ macroelement}\label{fg-p3}
\end{figure}
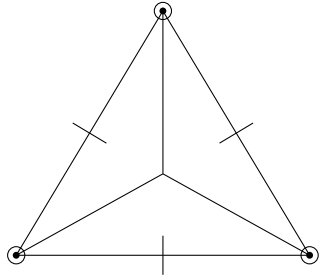
From the definition of $w$, we have
\begin{eqnarray}\label{w-norm}
	\|w\|_{L^2(D)}^2
	\approx
	h_D^2|w|_{H^2(D)}^2
	&\approx& 
	h_D\|w\|_{L^2(\partial D)}^2
	+h_D^3\left\|\frac{\partial w}{\partial n}\right\|_{L^2(\partial D)}^2 \nonumber \\
	&&= 
	h_D\|v\|_{L^2(\partial D)}^2
	+h_D^3\left\|\frac{\partial v}{\partial n}\right\|_{L^2(\partial D)}^2.
\end{eqnarray}
\begin{lemma}\label{inverse-inequality}
	$$
	|v|_{H^2(D)} \apprle h_D^{-2}|||v|||_{k,D},\ \forall v \in \mathcal{Q}^k(D).
	$$
\end{lemma}
\begin{proof}
	Following \cite{Brenner17}, it suffices to prove  when $h_D = 1$.
	Let $\phi$ be a smooth function supported on the disc $\mathfrak{B}$ with radius $\rho$,
	such that 
	$$
	\int_D\phi \ {\rm d}x= 1.
	$$
	By the
	equivalence of norms on finite dimensional spaces, we have
	$$
	\|p\|_{L^2(D)}\apprle \int_{D} \phi p^2 \ {\rm d}x,\ \forall p\in \mathbb{P}_{k}.
	$$  
	Let $w \in H^2 (D)$ be the piecewise polynomial constructed in Section \ref{case1}, and let $\xi = w + p\phi$ for $p\in \mathbb{P}_{k}$ such that
	$$
	\int_{D}(\xi-v)q\ {\rm d}x= 0, \ \forall q\in \mathbb{P}_{k},
	$$
	or equivalently
	\begin{equation}\label{pqphi}
		\int_{D}pq\phi \ {\rm d}x= \int_{D}(v-w)q\ {\rm d}x= \int_{D}(\Pi_{k,D}^0v -w)q \ {\rm d}x,\ \forall  q\in \mathbb{P}_{k}.
	\end{equation}
	Let $q=p$ in \eqref{pqphi}, then 
	$$
	\|p\|_{L^2(D)}\apprle \|\Pi_{k,D}^0v -w\|_{L^2(D)}\apprle  \|\Pi_{k,D}^0v\|_{L^2(D)}+\|w\|_{L^2(D)},
	$$
	and by Lemma \ref{leq-semi-norm},
	\begin{eqnarray*}
		\|\Pi_{k,D}^0v\|_{L^2(D)}^2
		&=&\|\Pi_{k-4,D}^0v\|_{L^2(D)}^2
		+\|(\Pi_{k,D}^0 -\Pi_{k-4,D}^0)v\|_{L^2(D)}^2\\
		&\apprle&\|\Pi_{k-4,D}^0v\|_{L^2(D)}^2
		+\|\Pi_{k,D}^\Delta v\|_{L^2(D)}^2\\
		&\apprle&\|\Pi_{k-4,D}^0v\|_{L^2(D)}^2
		+\|v\|_{L^2(\partial D)}^2
		+\left\|\frac{\partial v}{\partial n}\right\|_{L^2(\partial D)}^2.
	\end{eqnarray*}
	From \eqref{w-norm}, we have
	\begin{equation}\label{p2}
		\|p\|_{L^2(D)}^2
		\apprle 
		\|\Pi_{k-4,D}^0v\|_{L^2(D)}^2 
		+\|v\|_{L^2(\partial D)}^2
		+\left\|\frac{\partial v}{\partial n}\right\|_{L^2(\partial D)}^2.
	\end{equation}
	Also, by Lemma \ref{minimum-energy}, 
	$$
	|v|_{H^2(D)}\apprle |\xi|_{H^2(D)}.
	$$
	By \eqref{w-norm} and \eqref{p2},
	$$
	|\xi|_{H^2(D)}^2
	\apprle 
	|w|_{H^2(D)}^2+|p\phi|_{H^2(D)}^2
	\apprle 
	\|\Pi_{k-4,D}^0v\|_{L^2(D)}^2 
	+\|v\|_{L^2(\partial D)}^2
	+\left\|\frac{\partial v}{\partial n}\right\|_{L^2(\partial D)}^2.
	$$
	Then, $\forall v \in \mathcal{Q}^k(D) $, we have
	$$
	|v|_{H^2(D)}^2
	\apprle 
	\|\Pi_{k-4,D}^0v\|_{L^2(D)}^2 
	+\|v\|_{L^2(\partial D)}^2
	+\left\|\frac{\partial v}{\partial n}\right\|_{L^2(\partial D)}^2.
	$$
	By Lemma \ref{new-equi}, we arrived at the estimate.
\end{proof}
\begin{corollary}\label{corollary_2}
	For any $v\in \mathcal{Q}^k(D)$,
	$$
	\|v\|_{L^2(D)}+h_D|v|_{H^1(D)}+h^2_D|v|_{H^2(D)}
	\apprle
	|||v|||_{k,D}.
	$$
\end{corollary}
\begin{proof}
	From Lemma \ref{leq-semi-norm}, Corollary \ref{corollary_1} and Lemma \ref{inverse-inequality}, we have
	\begin{eqnarray*}
		\|v\|_{L^2(D)}
		&\apprle&
		\|v - \Pi_{k,D}^\Delta v\|_{L^2(D)}+\|\Pi_{k,D}^\Delta v\|_{L^2(D)}
		\apprle
		h_D^{2}|v|_{H^{2}(D)}+|||v|||_{k,D},
		\\
		h_D|v|_{H^{1}(D)}
		&\apprle&
		h_D|v - \Pi_{k,D}^\Delta v|_{H^{1}(D)}+h_D|\Pi_{k,D}^\Delta v|_{H^{1}(D)}
		\apprle
		h_D^{2}|v|_{H^{2}(D)}+|||v|||_{k,D}.
	\end{eqnarray*}
The proof is completed.
\end{proof}
\subsection{Estimate of Interpolation Operator}
The interpolation operator $I_{k,D} : H^3 (D) \rightarrow \mathcal{Q}^k(D)$ is
defined by the conditions that $\xi$ and $I_{k,D}\xi$ have the same value for each degree of freedom of $I_{k,D}\xi$.
It is clear that
$$
I_{k,D}\xi=\xi,\quad \forall \xi\in \mathcal{Q}^k(D)\ {\rm or}\ \forall \xi\in \mathbb{P}_k(D).
$$
\begin{lemma}\label{interpolation_error} The interpolation errors are listed below, $\forall \xi\in H^{l+1}(D),\ 1\leq l\leq k,$ we have
	\begin{eqnarray}
		\|\xi-I_{k,D}\xi\|_{L^2(D)}
		+\|\xi-\Pi_{k,D}^\Delta  I_{k,D}\xi\|_{L^2(D)}
		&\apprle&
		h_D^{l+1}|\xi|_{H^{l+1}(D)},\label{interpolation_error_l2}
		\\
		|\xi-I_{k,D}\xi|_{H^1(D)}
		+|\xi-\Pi_{k,D}^\Delta  I_{k,D}\xi|_{H^1(D)}
		&\apprle&
		h_D^{l}|\xi|_{H^{l+1}(D)},\label{interpolation_error_h1}
		\\
		|\xi-I_{k,D}\xi|_{H^2(D)}
		+|\xi-\Pi_{k,D}^\Delta  I_{k,D}\xi|_{H^2(D)}
		&\apprle&
		h_D^{l-1}|\xi|_{H^{l+1}(D)}.\label{interpolation_error_h2}
	\end{eqnarray}
\end{lemma}
\begin{proof}
	Suppose $h_D =1$, by Trace theorem, Lemma \ref{imbedding}, Lemma \ref{leq-semi-norm} and  Corollary \ref{corollary_2}, we have
	\begin{eqnarray*}
		\|I_{k,D}\xi\|_{L^2(D)}
		+\|\Pi_{k,D}^\Delta I_{k,D}\xi\|_{L^2(D)}
		&\apprle&
		|||I_{k,D}\xi|||_{k,D}
		\apprle
		\|\xi\|_{H^{l+1}(D)},
		\\
		|I_{k,D}\xi|_{H^1(D)}
		+|\Pi_{k,D}^\Delta I_{k,D}\xi|_{H^1(D)}
		&\apprle&
		|||I_{k,D}\xi|||_{k,D}
		\apprle
		\|\xi\|_{H^{l+1}(D)},
		\\
		|I_{k,D}\xi|_{H^2(D)}
		+|\Pi_{k,D}^\Delta I_{k,D}\xi|_{H^2(D)}
		&\apprle&
		|||I_{k,D}\xi|||_{k,D}
		\apprle
		\|\xi\|_{H^{l+1}(D)}.
	\end{eqnarray*}
	The results follow from Lemma \ref{bramble}, and $I_{k,D}q = q, \forall q\in \mathbb{P}_l$.
\end{proof}
\section{The Biharmonic Problem in Two Dimensions}
Let $\Omega$ be a bounded polygonal domain in $\mathbb{R}^2$, $f\in L^2(\Omega)$, the biharmonic equation is
\begin{equation}\label{biharmonic-equation}
	\begin{cases}
		\ \Delta^2u            &= f, \\
		\ u|_{\partial \Omega} &= 0, \\
		\ \left.\frac{\partial u}{\partial n}\right|_{\partial \Omega}  &= 0. 
	\end{cases}
\end{equation}
The variational formulation of \eqref{biharmonic-equation} is
finding $u\in H^2_0(\Omega)$, such that
$$
a(u,v) = (f,v),\ \forall v\in H^2_0(\Omega),
$$
where 
$$a(u,v) = ((u,v))_{\Omega}=\int_{\Omega}{\bf D}^2u\cdot{\bf D}^2v\ {\rm d}x.
$$
\begin{remark}
	For $u,v\in H^2_0(\Omega)$, we also have 
	$a(u,v) = \int_\Omega\Delta u\Delta v \  {\rm d}x$. However, $((u,v))_{D},$ where $D$ is a sub-domain of $\Omega$ and $\int_D\Delta u\Delta v \  {\rm d}x$ are not equivalent.
\end{remark}
In following sections, we will use virtual element method to solve \eqref{biharmonic-equation}.
\subsection{Virtual Element Spaces}
Let $\mathcal{T}_h$ be a conforming partition of $\Omega$ by polygonal
subdomains, i.e., the intersection of two distinct subdomains is either empty, common vertices
or common edges. We assume that all the polygons $D\in \mathcal{T}_h$ satisfy the shape regularity
assumptions in Section \ref{Shape-Regularity}.

We take the virtual element space $\mathcal{Q}_h^k$ to be 
$\{v \in H_0^2(\Omega): v|_D \in \mathcal{Q}^k(D),\ \forall D \in \mathcal{T}_h \}$ and
denote by $\mathcal{P}_{h}^k$ the space of (discontinuous) piecewise polynomials of degree $\leq k$ with respect
to $\mathcal{T}_h$. The operators are defined in terms of their local counterparts:
\begin{eqnarray}
	(\Pi_{k,h}^\Delta v)|_D  &=& \Pi_{k,D}^\Delta (v|_D),\quad \forall v\in  H^2(\Omega),\\
	(\Pi_{k,h}^{0}v)|_D  &=& \Pi_{k,D}^{0}(v|_D),\quad \forall v\in  L^2(\Omega),\\
	(I_{k,h}v)|_D  &=& I_{k,D}(v|_D),\quad \forall v\in  H^3(\Omega).
\end{eqnarray}
Also, the semi-norm on $H^2(\Omega)$ is defined as
\begin{eqnarray}
	|v|_{1,h}^2 = \sum_{D\in \mathcal{T}}|v|_{H^1(D)}^2,\quad |v|_{2,h}^2 = \sum_{D\in \mathcal{T}}|v|_{H^2(D)}^2,
\end{eqnarray}
so that $|v|_{2,h}=|v|_{H^2(D)}$ for $v\in H^2(D),$ and 
$$
|v-\Pi_{k,h}^{0}v|_{2,h} = \inf_{w\in \mathcal{P}_h^k}|v-w|_{H^2(D)},\ \forall v\in H^2(\Omega)+\mathcal{P}_h^k.
$$
{
	The local estimates: Corollary \ref{corollary_1}, \eqref{pikd0l2}-\eqref{pikd0h1}, Lemma \ref{pikd0-bound} and Lemma \ref{interpolation_error} immediately imply the following
	global results, where $h = \max\limits_{D\in\mathcal{T}_h} h_D$.}
\begin{corollary}\label{corollary_3} The global error estimates are listed below, $\forall \xi\in H^{l+1}(D),\ 1\leq l\leq k,$ we have
	\begin{eqnarray}
		\|\xi-I_{k,h}\xi\|
		&+&\|\xi-\Pi_{k,h}^\Delta  I_{k,h}\xi\|\nonumber \\
		&+&
		\|\xi - \Pi_{k,h}^0\xi\|+\|\xi - \Pi_{k,h}^\Delta\xi\|
		\apprle
		h^{l+1}|\xi|_{H^{l+1}(\Omega)},
		\\
		\|\xi-I_{k,h}\xi\|_{1,h}
		&+&\|\xi-\Pi_{k,h}^\Delta  I_{k,h}\xi\|_{1,h}\nonumber \\
		&+&
		|\xi - \Pi_{k,h}^0\xi|_{1,h}+|\xi - \Pi_{k,h}^\Delta\xi|_{1,h}
		\apprle
		h^{l}|\xi|_{H^{l+1}(\Omega)},
		\\
		\|\xi-I_{k,h}\xi\|_{2,h}
		&+&\|\xi-\Pi_{k,h}^\Delta  I_{k,h}\xi\|_{2,h}\nonumber \\
		&+&
		|\xi - \Pi_{k,h}^0\xi|_{2,h}+|\xi - \Pi_{k,h}^\Delta\xi|_{2,h}
		\apprle
		h^{l-1}|\xi|_{H^{l+1}(\Omega)},
	\end{eqnarray}
	where the norm $\|\cdot\|:= \|\cdot\|_{L^2(\Omega)}$.
\end{corollary}
\subsection{The Discrete Problem}
Our goal is to find $u_h\in \mathcal{Q}_h^k$, which satisfies
\begin{equation}\label{discrete-problem}
	a_h(u_h,v_h) = (f,\Xi_h v_h), \quad \forall v_h\in \mathcal{Q}_h^k,
\end{equation}
where $\Xi_h$ is an operator from $\mathcal{Q}_h^k$ to $\mathcal{P}_h^k$, and
\begin{eqnarray}
	a_h(w,v) &=& \sum_{D\in\mathcal{T}_h}
	\left(
	a^D(\Pi_{k,D}^\Delta w,\Pi_{k,D}^\Delta v)
	+S^D(w-\Pi_{k,D}^\Delta w,v-\Pi_{k,D}^\Delta v)
	\right),\label{discrete1} \\
	a^D(w,v)&=&\int_D{\bf D}^2w\cdot{\bf D}^2v\ {\rm d}x,\label{discrete2} \\
	S^D(w,v)&=&
	h_D^{-4}(\Pi^0_{k-4,D}w,\Pi^0_{k-4,D}v)_D
	+ 
	h_D^{-3}\sum_{e\in\mathcal{E}_D}
	(\Pi^0_{r,e}w,\Pi^0_{r,e}v)_e \nonumber \\ 
	&&\quad 
	+
	h_D^{-1}\sum_{e\in\mathcal{E}_D,i=1,2}
	\left(
	\Pi^0_{r-1,e}\frac{\partial w}{\partial x_i},
	\Pi^0_{r-1,e}\frac{\partial v}{\partial x_i}
	\right)_e,\label{discrete3} 
\end{eqnarray}
for $w,v\in H^2(\Omega)$. If $w,v\in \mathcal{Q}^k(D),$ we have
$$
\Pi^0_{r,e}w = w|_e,\quad \Pi^0_{r-1,e}\frac{\partial w}{\partial x_i} = \left.\frac{\partial w}{\partial x_i}\right|_e,
$$
so that $S^D(w,v)$ can be computed explicitly with the degrees of freedom of $\mathcal{Q}^k(D)$.
\subsubsection{Other Choices of $S^D(\cdot,\cdot)$}
The systems of virtual element method are equivalent if the bilinear form satisfies 
$$
S^D(v,v) \approx h^{-4}_D|||v|||_{k,D}^2,\quad \forall v\in {\rm Ker}\Pi_{k,D}^\Delta .
$$
From Lemma \ref{equivalence_b} and Remark \ref{remark_t}, we can take 
\begin{eqnarray}
	S^D(w,v)&=& 
	h_D^{-3}
	(w,v)_{\partial D}
	+
	h_D^{-1} 
	\left(
	\frac{\partial w}{\partial n},
	\frac{\partial v}{\partial n}
	\right)_{\partial D},\\
	S^D(w,v)&=&
	h_D^{-1} 
	\left(
	\frac{\partial w}{\partial t},
	\frac{\partial v}{\partial t}
	\right)_{\partial D}
	+
	h_D^{-1} 
	\left(
	\frac{\partial w}{\partial n},
	\frac{\partial v}{\partial n}
	\right)_{\partial D},
\end{eqnarray}
for $w, v\in {\rm Ker}\Pi_{k,D}^\Delta $.
\subsection{Well-posedness of the Discrete Problem}
We can show the well-posedness by the following Lemmas.
\begin{lemma}\label{S^Dbound2}
	For any $v\in H^2(D)$, we have
	$$
	S^D(v-\Pi_{k,D}^\Delta v,v-\Pi_{k,D}^\Delta v)^{1/2}\apprle
	|v-\Pi_{k,D}^\Delta v|_{H^2(D)}.
	$$
\end{lemma}
\begin{proof}
	By \eqref{pikd1}-\eqref{pikd3} and Lemma \ref{friedrichs}, let $w = v-\Pi_{k,D}^\Delta v$, we have $\bar{w} = 0$, and
	$$
	S^D(w,w)^{1/2} = h_D^{-2}|||w-\bar{w}|||_{k,D}\apprle
	|w|_{H^2(D)}.
	$$
\end{proof}
\begin{lemma}\label{S^D-equivalent}
	For any $v\in {\rm Ker}\Pi_{k,D}^\Delta $, we have
	$S^D(v,v)\approx |v|^2_{H^2(D)}.$
\end{lemma}
\begin{proof}
	For any $v\in {\rm Ker}\Pi_{k,D}^\Delta,$ we have $v\in H^2(D)$ and
	$$
	v = v-\Pi_{k,D}^\Delta v,
	$$	
	with Lemma \ref{inverse-inequality} and Lemma \ref{S^Dbound2}, we have the equivalence.
\end{proof}
\begin{remark}\label{S^Dbound}
	By Lemma \ref{S^Dbound2}, for any $v,w\in \mathcal{Q}_h^k$, we have
	$$
	S^D(v-\Pi_{k,D}^\Delta v,w-\Pi_{k,D}^\Delta w)\apprle
	|v-\Pi_{k,D}^\Delta v|_{H^2(D)} |w-\Pi_{k,D}^\Delta w|_{H^2(D)}.
	$$	
\end{remark}
Then by Lemma \ref{S^D-equivalent}, as in \cite{Brenner17}, we have 
\begin{equation}\label{a(vv)}
	a_h(v,v) \approx a(v,v),\quad \forall v\in \mathcal{Q}_h^k,
\end{equation}
which means problem \eqref{discrete-problem} is uniquely solvable.
\subsection{Choice of $\Xi_h$}
Here, we chose $\Xi_h$ as
\begin{equation}
	\Xi_h =
	\begin{cases}
		\Pi_{k,h}^0       & \quad \text{if } 2\leq k \leq 3,\\
		\Pi_{k-1,h}^0     & \quad \text{if } k\geq 4.
	\end{cases}
\end{equation}
The following result can be used for error analysis in $H^2(\Omega)$ norm. Define $H^0(\Omega) := L^2(\Omega)$, and the indices $(k, l, m)$ as
\begin{equation}\label{index_klm}
	(k, l, m):=\begin{cases}
		{\rm for}\ 2\leq k\leq 3,\ l=0,\ m=2,\\
		{\rm for}\ 4\leq k,\ l = k-3,\ m=2+l.
	\end{cases}
\end{equation}
\begin{lemma}\label{klm-h2}
	With \eqref{index_klm},  we have
	\begin{equation}
		(f,w-\Xi_h w)  \apprle
		h^{2+l}|f|_{H^l(\Omega)} |w|_{H^2(\Omega)},       
		\quad 
		\forall f\in H^{l}(\Omega),\ w\in \mathcal{Q}_h^k.
	\end{equation}
\end{lemma} 
\begin{proof}
	For $k\leq 3, f\in L^2(\Omega)$, we define $\Pi_{k-4}^0 f = 0,$ so that with Corollary \ref{corollary_3}, we have
	\begin{eqnarray*}
		(f,w-\Xi_h w) 
		&=& (f-\Pi_{k-4,h}^0 f,w-\Xi_h w)\\
		&=& (f,w-\Pi_{k,h}^0 w)\\
		&\leq& 
		\|f\|_{L^2(\Omega)} \|w-\Pi_{1,h}^0 w\|_{L^2(\Omega)}\\
		&\apprle& h^2\|f\|_{L^2(\Omega)} |w|_{H^2(\Omega)}
	\end{eqnarray*}
	For $k\geq 4, f\in H^l(\Omega), l = k-3$,  with Lemma \ref {bramble} and Corollary \ref{corollary_3}, we have
	\begin{eqnarray*}
		(f,w-\Xi_h w) 
		&=& (f-\Pi_{k-4,h}^0 f,w-\Xi_h w)\\
		&\leq& 
		\|f-\Pi_{k-4,h}^0 f\|_{L^2(\Omega)} \|w-\Pi_{1,h}^0 w\|_{L^2(\Omega)}\\
		&\apprle& h^{2+l}|f|_{H^l(\Omega)} |w|_{H^2(\Omega)}
	\end{eqnarray*}
\end{proof}
\begin{lemma}\label{klm-h3}
	With \eqref{index_klm},  we have
	\begin{equation}\label{klm-h3-1}
		(f,I_{k,h}w-\Xi_h I_{k,h}w)  \apprle
		h^{3+l}|f|_{H^l(\Omega)} |w|_{H^3(\Omega)},       
	\end{equation}
for any $f\in H^{l}(\Omega),\ w\in H^3(\Omega),\ k\geq 2.$ And 

	\begin{equation}\label{fik=3h}
		(f,I_{k,h}w-\Xi_h I_{k,h}w)  \apprle
		h^{3+s+l}|f|_{H^l(\Omega)} \|w\|_{H^{3+s}(\Omega)},      
	\end{equation}
for any $f\in H^{l}(\Omega),\ w\in H^{3+s}(\Omega),\ k\geq 3, \ 0<s\leq 1.$
\end{lemma} 
\begin{proof}
	Follow the proof in Lemma \ref{klm-h2} with Lemma \ref{pikd0-bound} and Lemma \ref{interpolation_error}.
	For $k\geq 2$, $w\in H^3(\Omega)$,  we have
	\begin{eqnarray*}
		(f,I_{k,h} w -\Xi_h I_{k,h} w) 
		&=& (f-\Pi_{k-4,h}^0f,I_{k,h} w -\Xi_h I_{k,h} w)\\
		&=& (f-\Pi_{k-4,h}^0f,I_{k,h} w -w + w-\Xi_h w + \Xi_h (w-I_{k,h} w) ),
	\end{eqnarray*}
	then with Corollary \ref{corollary_3}, we can get \eqref{klm-h3-1},
	so as $k\geq 3$ and $w\in H^{3+s}(\Omega)$.
	
\end{proof}
\subsection{Error Estimate in $|\cdot|_{H^2(\Omega)}$ Norm}
Firstly, the error estimate in $|\cdot|_{H^2(\Omega)}$ norm for $u-u_h$ is given in Theorem \ref{theorem1}.
\begin{theorem}\label{theorem1}
	With \eqref{index_klm}, we have
	$$
	|u-u_h|_{H^2(\Omega)}
	\apprle
	\inf_{v_h\in\mathcal{Q}_h^k}|u-v_h|_{H^2(\Omega)}
	+
	\inf_{p\in\mathcal{P}_h^k}|u-p|_{2,h}
	+
	h^m|f|_{H^l(\Omega)}.
	$$
	Suppose $u\in H^{k+1}(\Omega)$ then
	$$
	|u-u_h|_{H^2(\Omega)}
	\apprle
	h^{k-1}(|u|_{H^{k+1}(\Omega)}+ |f|_{H^{l}(\Omega)}).
	$$
\end{theorem}
\begin{proof}
	Similar as in \cite{Brenner17}, for any given $v_h\in \mathcal{Q}_h^k$, from \eqref{a(vv)}, we have
	$$
	|u_h-v_h|_{H^2(\Omega)}\apprle 
	\max_{w_h\in\mathcal{Q}_h^k}
	\frac{a_h(v_h-u_h,w_h)}
	{|w_h|_{H^2(\Omega)}},
	$$
	and by \eqref{discrete-problem}, 
	$$
	a_h(v_h-u_h,w_h) = a_h(v_h,w_h)-(f,\Xi_h w_h).
	$$
	Then from \eqref{discrete1} to \eqref{discrete3}, 
	by \eqref{h2_error}, \eqref{interpolation_error_h2},  we have
	\begin{eqnarray*}
		a_h(v_h-u_h,w_h) 
		&=& 
		\sum_{D\in\mathcal{T}_h}
		a^D(\Pi_{k,D}^\Delta (v_h -u)+(\Pi_{k,D}^\Delta u-u),w_h)
		+
		(f,w_h-\Xi_h w_h)\\
		&&+
		\sum_{D\in\mathcal{T}_h} 
		S^D(
		(I-\Pi_{k,D}^\Delta )(v_h-\Pi_{k,D}^\Delta  u),(I-\Pi_{k,D}^\Delta )w_h
		),
	\end{eqnarray*}
	with Lemma \ref{klm-h2}, Remark \ref{S^Dbound} or Lemma \ref{S^Dbound2} and Corollary \ref{corollary_3}, the estimate is obtained.
\end{proof}
\subsection{Error Estimate in $|\cdot|_{H^1(\Omega)}$ and $\|\cdot\|_{L^2(\Omega)}$ Norm}
We suppose $\Omega$ is also convex and start with a consistency result.
\begin{lemma}
	Suppose $u\in H^{k+1}(\Omega)$ and $l$ is defined in \eqref{index_klm}, then
	\begin{equation}\label{h166}
		a(u-u_h,I_{k,h}\xi)\apprle h^{k}(|u|_{H^{k+1}(\Omega)}+ |f|_{H^{l}(\Omega)})|\xi|_{H^3(\Omega)} ,
	\end{equation}
for any $\xi\in H^3(\Omega)\cap H^2_0(\Omega),\ k\geq 2.$ And 
	\begin{equation}\label{l267}
		a(u-u_h,I_{k,h}\xi)\apprle h^{k+s}(|u|_{H^{k+1}(\Omega)}+ |f|_{H^{l}(\Omega)})\|\xi\|_{H^{3+s}(\Omega)} ,
	\end{equation}
for any $\xi\in H^{3+s}(\Omega)\cap H^2_0(\Omega),\ k\geq 3$.
\end{lemma}
\begin{proof}
	Similar as in \cite{Brenner17}, we have
	\begin{eqnarray*}
		a(u-u_h,I_{k,h}\xi) 
		&=& 
		\sum_{D\in\mathcal{T}_h}
		a^D(u-u_h, I_{k,D}\xi - \Pi_{k,D}^\Delta I_{k,D}\xi)\\
		&&+
		\sum_{D\in\mathcal{T}_h}
		a^D(\Pi_{k,D}^\Delta u-u, I_{k,D}\xi - \Pi_{k,D}^\Delta I_{k,D}\xi)\\
		&&+
		\sum_{D\in\mathcal{T}_h} 
		S^D(
		(I-\Pi_{k,D}^\Delta )u_h,(I-\Pi_{k,D}^\Delta )I_{k,D}\xi
		)\\
		&&+
		(f,I_{k,h}\xi-\Xi I_{k,h}\xi)
		\\
		&\apprle&
		\left(
		|u-u_h|_{H^2(\Omega)}+|u-\Pi_{k,h}^\Delta u|_{2,h}
		\right)
		|I_{k,h}\xi - \Pi_{k,h}^\Delta I_{k,h}\xi|_{2,h}\\
		&&+(f,I_{k,h}\xi-\Xi I_{k,h}\xi).
	\end{eqnarray*}
	Then by Lemma \ref{klm-h3}, Corollary \ref{corollary_3}, and Theorem \ref{theorem1}, we get the estimate.
\end{proof}
From the regularity results of \eqref{biharmonic-equation}, see \cite{Grisvard92}, we have
\begin{equation}\label{regularity-h1}
	\|u\|_{H^3(\Omega)} \apprle \|f\|_{H^{-1}(\Omega)},\quad \forall f\in {H^{-1}(\Omega)}.
\end{equation}
\begin{equation}\label{regularity-l2}
	\|u\|_{H^{3+s}(\Omega)} \apprle \|f\|_{L^2(\Omega)},\quad \forall f\in {L^2(\Omega)}, \ 0<s\leq 1.
\end{equation}
\begin{theorem}\label{theorem2}
	Suppose $u\in H^{k+1}(\Omega)$ then
	\begin{equation}
		|u-u_h|_{H^1(\Omega)}\apprle h^{k}(|u|_{H^{k+1}(\Omega)}+ |f|_{H^{l}(\Omega)}),
	\end{equation}	
	where $l$ is defined in \eqref{index_klm}.
\end{theorem}
\begin{proof}
	Using the duality arguments and \eqref{h166}, let $f=-\Delta(u-u_h),$ and $\phi\in H_0^2(\Omega)$ be the solution of \eqref{biharmonic-equation}, then we get
	\begin{eqnarray*}
		|u-u_h|^2_{H^1(\Omega)}
		&=&(\Delta^2\phi,u-u_h)\\
		&=&a(u-u_h,\phi-I_{k,h}\phi)+a(u-u_h,I_{k,h}\phi)\\
		&\apprle&
		h^{k}(|u|_{H^{k+1}(\Omega)}+ |f|_{H^{l}(\Omega)})|\phi|_{H^3(\Omega)},
	\end{eqnarray*}
	by \eqref{regularity-h1}, we have
	$|\phi|_{H^3(\Omega)}\apprle |u-u_h|_{H^1(\Omega)}$, then the result is obtained.
\end{proof}
\begin{theorem}\label{theorem3}
	Suppose $u\in H^{k+1}(\Omega)$ then
	\begin{equation}\label{l2k=2}
		\|u-u_h\|_{L^2(\Omega)}\apprle h^{2}(|u|_{H^{3}(\Omega)}+ |f|_{L^2(\Omega)}), \ k=2,
	\end{equation}	
	\begin{equation}
		\|u-u_h\|_{L^2(\Omega)}\apprle h^{k+s}(|u|_{H^{k+1}(\Omega)}+ |f|_{H^{l}(\Omega)}),\ k\geq 3, \ 0<s\leq 1,
	\end{equation}	
	where $l$ is defined in \eqref{index_klm}.
\end{theorem}
\begin{proof}
	For $k=2,$ by Theorem \ref{theorem2} and Poincar$\acute{\rm e}$ inequality, \eqref{l2k=2} is obtained.
	For $k\geq 3$, using the duality arguments and \eqref{l267}, let $f=u-u_h,$ and $\phi\in H_0^2(\Omega)$ be the solution of \eqref{biharmonic-equation}, then we get
	\begin{eqnarray*}
		\|u-u_h\|^2_{L^2(\Omega)}
		&=&(\Delta^2\phi,u-u_h)\\
		&=&a(u-u_h,\phi-I_{k,h}\phi)+a(u-u_h,I_{k,h}\phi)\\
		&\apprle&
		h^{k+s}(|u|_{H^{k+1}(\Omega)}+ |f|_{H^{l}(\Omega)})\|\phi\|_{H^{3+s}(\Omega)},
	\end{eqnarray*}
	by \eqref{regularity-l2}, we have
	$\|\phi\|_{H^{3+s}(\Omega)}\apprle \|u-u_h\|_{L^2(\Omega)}$, then the result is obtained.
\end{proof}
\subsection{Error Estimates for $\Pi_{k,h}^\Delta u_h$}
We also have an error estimate for the computable 
$\Pi_{k,h}^\Delta u_h$.
\begin{corollary}
	Suppose $u\in H^{k+1}(\Omega)$ then
	$$
	|u-\Pi_{k,h}^\Delta u_h|_{2,h}
	\apprle
	h^{k-1}(|u|_{H^{k+1}(\Omega)}+ |f|_{H^{l}(\Omega)}).
	$$
\end{corollary}
\begin{proof}
	By \eqref{obvious_ineq}, Theorem \ref{theorem1}, Corollary \ref{corollary_3}, and 
	$$
	|u-\Pi_{k,h}^\Delta u|_{2,h}\leq \inf_{p\in\mathcal{P}_h^k}|u-p|_{2,h}, \ \forall u\in H^2(\Omega),
	$$
	the estimate is obtained.
\end{proof}
\begin{corollary}\label{coh1}
	Suppose $u\in H^{k+1}(\Omega)$ then
	$$
	|u-\Pi_{k,h}^\Delta u_h|_{1,h}
	\apprle
	h^{k}(|u|_{H^{k+1}(\Omega)}+ |f|_{H^{l}(\Omega)}),
	$$
	where $l$ is defined in \eqref{index_klm}.
\end{corollary}
\begin{proof}
	By Corollary \ref{corollary_1}, Corollary \ref{corollary_3}, Theorem \ref{theorem2},   Lemma \ref{leq-semi-norm} and
	$$
	|u-\Pi_{k,h}^\Delta u_h|_{1,h}\leq |u-\Pi_{k,h}^\Delta u|_{1,h}+|\Pi_{k,h}^\Delta u-\Pi_{k,h}^\Delta u_h|_{1,h}.
	$$
	Suppose $\xi = u-u_h,$ and $\bar{\xi}$ is defined as in Lemma \ref{friedrichs},
	the second term is estimated as
	$$
	h_D|\Pi_{k,D}^\Delta \xi|_{H^1(D)}=
	h_D|\Pi_{k,D}^\Delta (\xi-\bar{\xi})|_{H^1(D)}\apprle |||\xi-\bar{\xi}|||_{k,D}\apprle 
	h_D^2|\xi|_{H^2(D)},
	$$
	so that
	$$
	|\Pi_{k,D}^\Delta (u-u_h)|_{H^1(D)}^2\apprle
	h^2_D|u-u_h|^2_{H^2(D)}
	$$
	sum them up then the estimate is obtained.
\end{proof}
\begin{corollary}
	Suppose $u\in H^{k+1}(\Omega)$ then
	\begin{equation}\label{hl2k=2}
		\|u-\Pi_{k,h}^\Delta u_h\|_{L^2(\Omega)}\apprle h^{2}(|u|_{H^{3}(\Omega)}+ \|f\|_{L^2(\Omega)}), \ k=2,
	\end{equation}	
	\begin{equation}
		\|u-\Pi_{k,h}^\Delta u_h\|_{L^2(\Omega)}\apprle h^{k+s}(|u|_{H^{k+1}(\Omega)}+ |f|_{H^{l}(\Omega)}),\ k\geq 3, \ 0<s\leq 1,
	\end{equation}	
	where $l$ is defined in \eqref{index_klm}.
\end{corollary}
\begin{proof}
	By Corollary \ref{corollary_1}, Corollary \ref{corollary_3}, Theorem \ref{theorem2}, Theorem \ref{theorem3},  Lemma \ref{leq-semi-norm} and
	$$
	\|u-\Pi_{k,h}^\Delta u_h\|_{L^2(\Omega)}\leq \|u-\Pi_{k,h}^\Delta u\|_{L^2(\Omega)}+\|\Pi_{k,h}^\Delta u-\Pi_{k,h}^\Delta u_h\|_{L^2(\Omega)}.
	$$
	From Lemma \ref{friedrichs},
	the second term is estimated as
	$$
	\|\Pi_{k,D}^\Delta (u-u_h)\|_{L^2(D)}^2
	\apprle 
	\|u-u_h\|_{L^2(D)}^2+
	h_D^2|u-u_h|_{H^1(D)}^2
	+h_D^4|u-u_h|_{H^2(D)}^2,
	$$
	sum them up then, the estimate is obtained.
\end{proof}
\subsection{Error Estimates for $\Pi_{k,h}^0u_h$}
Since $\Pi_{k,h}^0u_h$ can be computed explicitly, we can also get the similar error estimates between $u$ and $\Pi_{k,h}^0u_h$.
\begin{corollary}\label{coh2pik0}
	Suppose $u\in H^{k+1}(\Omega)$ then
	$$
	|u-\Pi_{k,h}^0 u_h|_{2,h}
	\apprle
	h^{k-1}(|u|_{H^{k+1}(\Omega)}+ |f|_{H^{l}(\Omega)}).
	$$
\end{corollary}
\begin{proof}
	By Lemma \ref{pikd0-bound}, Theorem \ref{theorem1}, and 
	$$
	|u-\Pi_{k,h}^0 u_h|_{2,h}\leq |u-\Pi_{k,h}^0 u|_{2,h}
	+|\Pi_{k,h}^0 (u-u_h)|_{2,h}, \ \forall u\in H^2(\Omega),
	$$
	the estimate is obtained.
\end{proof}

\begin{corollary}\label{coh1-pik0}
	Suppose $u\in H^{k+1}(\Omega)$ then
	$$
	|u-\Pi_{k,h}^0 u_h|_{1,h}
	\apprle
	h^{k}(|u|_{H^{k+1}(\Omega)}+ |f|_{H^{l}(\Omega)}),
	$$
	where $l$ is defined in \eqref{index_klm}.
\end{corollary}
\begin{proof}
	By \eqref{pikd0l2}, \eqref{pikd0h1}, Theorem \ref{theorem2}, and
	$$
	|u-\Pi_{k,h}^0 u_h|_{1,h}\leq |u-\Pi_{k,h}^0 u|_{1,h}+|\Pi_{k,h}^0 (u-u_h)|_{1,h},
	$$ 
	for the second term, we have
	$$
	|\Pi_{k,D}^0 (u-u_h)|_{H^1(D)}^2\apprle
	|u-u_h|^2_{H^1(D)}
	$$
	sum them up then the estimate is obtained.
\end{proof}

\begin{corollary}\label{copik0l2}
	Suppose $u\in H^{k+1}(\Omega)$ then
	\begin{equation}\label{hl2k=2-pik0}
		\|u-\Pi_{k,h}^0 u_h\|_{L^2(\Omega)}\apprle h^{2}(|u|_{H^{3}(\Omega)}+ \|f\|_{L^2(\Omega)}), \ k=2,
	\end{equation}	
	\begin{equation}
		\|u-\Pi_{k,h}^0 u_h\|_{L^2(\Omega)}\apprle h^{k+s}(|u|_{H^{k+1}(\Omega)}+ |f|_{H^{l}(\Omega)}),\ k\geq 3, \ 0<s\leq 1,
	\end{equation}	
	where $l$ is defined in \eqref{index_klm}.
\end{corollary}
\begin{proof}
	By Theorem \ref{theorem3}, and
	$$
	\|u-\Pi_{k,h}^0 u_h\|_{L^2(\Omega)}\leq \|u-\Pi_{k,h}^0 u\|_{L^2(\Omega)}+\|\Pi_{k,h}^0 (u-u_h)\|_{L^2(\Omega)},
	$$
	the second term is estimated as
	$$
	\|\Pi_{k,D}^0 (u-u_h)\|_{L^2(\Omega)}
	\apprle 
	\|u-u_h\|_{L^2(\Omega)},
	$$
	so the results are obtained.
\end{proof}

\section{Conclusion}
We have extended the works done in \cite{Brenner17} to forth order problems in two dimension. Similar basic estimates for local projections $\Pi_{k,D}^\Delta$, $\Pi_{k,D}^0$, $I_{k,D}$  and the improved error analysis of modified virtual element method for biharmonic equation are obtained. The computable piecewise polynomials $\Pi_{k,h}^\Delta u_h$ and 
$\Pi_{k,h}^0 u_h$ are more efficient to use in practice. 

We can replace \eqref{pikd2} by \eqref{pikd4}
\begin{equation}\label{pikd4}
	\int_D \nabla \Pi_{k,D}^\Delta \xi\ {\rm d}x
	= \int_D \nabla  \xi\ {\rm d}x,
\end{equation}
To compute \eqref{pikd4}, we have 
\begin{equation}
	\int_D\frac{\partial \xi }{\partial x_i}\ {\rm d}x
	= \int_{\partial D}{\xi }n_i\ {\rm d}s.
\end{equation}
For $k\geq 4$, can replace \eqref{pikd3} by \eqref{pikd5}
\begin{equation}\label{pikd5}
	\int_{D} \Pi_{k,D}^\Delta \xi\ {\rm d}s
	= \int_{D}   \xi\ {\rm d}x.
\end{equation}
And the replacements attain same estimates for projections and error analysis.

\section*{Acknowledgments} We would like to thank anonymous reviewers for their valuable comments and suggestions. Also we are grateful for the editors' patience while handling this paper.

\medskip
Received xxxx 20xx; revised xxxx 20xx.
\medskip

\end{document}